\documentclass[sn-mathphys,Numbered]{sn-jnl}

\usepackage{graphicx}%
\usepackage{multirow}%
\usepackage{amsmath,amssymb,amsfonts}%
\usepackage{amsthm}%
\usepackage{mathrsfs}%
\usepackage[title]{appendix}%
\usepackage{xcolor}%
\usepackage{textcomp}%
\usepackage{manyfoot}%
\usepackage{booktabs}%
\usepackage{algorithm}%
\usepackage{algorithmicx}%
\usepackage{algpseudocode}%
\usepackage{listings}%

\theoremstyle{thmstyleone}%
\newtheorem{theorem}{Theorem}
\newtheorem{corollary}{Corollary}
\newtheorem{lemma}{Lemma}
\newtheorem{proposition}{Proposition}%

\theoremstyle{thmstyletwo}%
\newtheorem{example}{Example}%
\newtheorem{remark}{Remark}%

\theoremstyle{thmstylethree}%
\newtheorem{definition}{Definition}%

\begin{document}
	
	\title[On $K$-frames for quaternionic Hilbert spaces]{On $K$-frames for Quaternionic Hilbert spaces}
	
	
	\author{\fnm{Najib} \sur{Khachiaa}}\email{khachiaa.najib@uit.ac.ma}

	\affil{\orgdiv{Laboratory Partial Differential Equations, Spectral Algebra and Geometry, Department of Mathematics}, \orgname{Faculty of Sciences, University Ibn Tofail}, \orgaddress{\city{Kenitra}, \country{Morocco}}}

	\abstract{The aim of this paper is to study \( K \)-frames for quaternionic Hilbert spaces. First, we present the quaternionic version of Douglas's theorem and then investigate \( K \)-frames for a quaternionic Hilbert space \( \mathcal{H} \), where \( K \in \mathbb{B}(\mathcal{H}) \). Given two quaternionic Hilbert spaces \( \mathcal{H}_1 \) and \( \mathcal{H}_2 \), along with two right \(\mathbb{H}\)-linear bounded operators \( K_1 \in \mathbb{B}(\mathcal{H}_1) \) and \( K_2 \in \mathbb{B}(\mathcal{H}_2) \), we study the \( K_1 \oplus K_2 \)-frames for the super space \( \mathcal{H}_1 \oplus \mathcal{H}_2 \) and their relationship with \( K_1 \)-frames and \( K_2 \)-frames for \( \mathcal{H}_1 \) and \( \mathcal{H}_2 \), respectively. We also explore the \( K_1 \oplus K_2 \)-duality in relation to \( K_1 \)-duality and \( K_2 \)-duality. }


	
	\maketitle

\section{Introduction and preliminaries}
Frames in quaternionic Hilbert spaces provide a robust framework for analyzing and reconstructing signals within a higher-dimensional space. As an extension of the classical frame theory, these structures facilitate efficient data representation and processing. The unique properties of quaternionic spaces offer new avenues for exploration, particularly in applications such as signal processing and communications. Recently, the theory of frames has undergone several generalizations, including g-frames and $K$-frames. These advancements have broadened the applicability of frame theory in various mathematical contexts. In this work, we will study the theory of $K$-frames for quaternic Hilbert spaces and  will also study a particular case, which is the theory of $K$-frames for  quaternionic super Hilbert spaces. The quaternionic field is an extension of the real and complex number systems, consisting of numbers known as quaternions. Quaternions are used to represent three-dimensional rotations and orientations, making them invaluable in computer graphics and robotics. Unlike real and complex numbers, quaternion multiplication is non-commutative, which adds complexity to their algebraic structure. Quaternions also provide a more efficient way to perform calculations in three-dimensional space, enhancing applications in physics and engineering.

\begin{definition}[the field of quaternions]\cite{11}
The non-commutative field of quaternions $\mathbb{H}$ is a four-dimensional real algebra with unity. In $\mathbb{H}$, $0$ denotes the null element  and $1$ denotes the identity with respect to multiplication. It also includes three so-called imaginary units, denoted by $i,j,k$. i.e.,
$$\mathbb{H}=\{a_0+a_1i+a_2j+a_3k:\; a_0,a_1,a_2,a_3\in \mathbb{R}\},$$
where $i^2=j^2=k^2=-1$, $ij=-ji=k$, $jk=-kj=i$ and $ki=-ik=j$. For each quaternion $q=a_0+a_1i+a_2j+a_3k$, we deifine the conjugate of $q$ denoted by $\overline{q}=a_0-a_1i-a_2j-a_2k \in \mathbb{H}$ and the module of $q$ denoted by $\vert q\vert $ as 
$$\vert q\vert =(\overline{q}q)^{\frac{1}{2}}=(q\overline{q})^{\frac{1}{2}}=\displaystyle{\sqrt{a_0^2+a_1^2+a_2^2+a_3^2}}.$$
For every $q\in \mathbb{H}$, $q^{-1}=\displaystyle{\frac{\overline{q}}{\vert q\vert ^2}}.$
\end{definition}

\begin{definition}[ Right quaternionic vector space]\cite{11}
A right quaternioniq vector space $V$ is a linear vector space under right scalar multiplication over the field of quaternions $\mathbb{H}$, i.e., the right scalar multiplication 
$$\begin{array}{rcl}
V\times \mathbb{H} &\rightarrow& V\\
(v,q)&\mapsto& v.q,
\end{array}$$
satisfies the following for all $u,v\in V$ and $q,p\in \mathbb{H}$:
\begin{enumerate}
\item $(v+u).q=v.q+u.q$,
\item $v.(p+q)=v.p+v.q$,
\item $v.(pq)=(v.p).q$.
\end{enumerate}
Instead of $v.q$, we often use the notation $vq$.
\end{definition}

\begin{definition}[Right quaterninoic pre-Hilbert space]\cite{11}
A right quaternionic pre-Hilbert space $\mathcal{H}$, is a right quaternionic vector space equipped with the binary mapping\\ \( \langle \cdot \mid \cdot \rangle : \mathcal{H} \times \mathcal{H} \to \mathbb{H} \) (called the Hermitian quaternionic inner product) which satisfies the following properties:

\begin{itemize}
    \item[(a)] \( \langle v_1 \mid v_2 \rangle = \overline{\langle v_2 \mid v_1 \rangle} \) for all \( v_1, v_2 \in \mathcal{H}\),
    \item[(b)] \( \langle v \mid v \rangle > 0 \) if \( v \neq 0 \),
    \item[(c)] \( \langle v \mid v_1 + v_2 \rangle = \langle v \mid v_1 \rangle + \langle v \mid v_2 \rangle \) for all \( v, v_1, v_2 \in \mathcal{H} \),
    \item[(d)] \( \langle v \mid uq \rangle = \langle v \mid u \rangle q \) for all \( v, u \in \mathcal{H} \) and \( q \in \mathbb{H} \).
\end{itemize}
\end{definition}

In view of Definition  $3$ , a right pre-Hilbert space $\mathcal{H}$ also has the property:

\begin{itemize}
    \item[(i)] $ \langle vq \mid u \rangle = \overline{q} \langle v \mid u \rangle$ for all $v, u \in \mathcal{H} $ and \( q \in \mathbb{H} \).
\end{itemize}

Let \( \mathcal{H} \) be a right quaternionic pre-Hilbert space with the Hermitian inner product \( \langle \cdot \mid \cdot \rangle \). Define the quaternionic norm \( \|\cdot\| : \mathcal{H} \to \mathbb{R}^+ \) on \( \mathcal{H} \) by
\[
\|u\| = \sqrt{\langle u \mid u \rangle}, \quad u \in \mathcal{H}, 
\]
which satisfies the following properties:
\begin{enumerate}
\item $\|uq\|=\|u\|\vert q\vert$, for all $u\in \mathcal{H}$ and $q\in \mathbb{H}$,
\item $\| u+v\|\leq \|u\|+\|v\|$,
\item $\|u\|=0\Longleftrightarrow u=0$ for $u\in \mathcal{H}$.
\end{enumerate}
\begin{definition}[ Right quaternionic Hilbert space]\cite{11}
A right quaternionic pre-Hilbert space is called a right quaternionic Hilbert space if it is complete with respect to the quaternionic norm.
\end{definition}

\begin{example}
Define $$\ell^2(\mathbb{H}):=\left\{ \{q_i\}_{i\in I}\subset \mathbb{H}:\; \displaystyle{\sum_{i\in I}\vert q_i\vert^2<\infty} \right\}.$$
$\ell^2(\mathbb{H})$ under right multiplication by quaternionic scalars together with the quaternionic inner product defined as: $\langle p\mid q\rangle:=\displaystyle{\sum_{i\in I} \overline{p_i}q_i}$ for  $p=\{p_i\}_{i\in I}$ and $q=\{q_i\}_{i\in I}\in \ell^2(\mathbb{H})$, is a right quaternionic Hilbert space.
\end{example}
\begin{theorem}[The Cauchy-Schwarz inequality]\cite{11}
If $\mathcal{H}$ is a right quaternionic Hilbert space, then for all $u,v\in \mathcal{H}$, 
$$\vert \langle u\mid v\rangle \vert\leq \|u\| \|v\|.$$
\end{theorem}
\begin{definition}[orthogonality]\cite{11}
Let \(\mathcal{H} \) be a right quaternionic Hilbert space and \( A\) be a subset of \( \mathcal{H} \). Then, define the set:

\begin{itemize}
    \item \( A^{\perp} = \{ v \in \mathcal{H} : \langle v \mid u \rangle = 0 \; \forall \; u \in A \} \);
    \item \( \langle A \rangle \) as the right quaternionic vector subspace of \( \mathcal{H} \) consisting of all finite right \( \mathbb{H} \)-linear combinations of elements of \( A\).
\end{itemize}
\end{definition}

\begin{proposition}\cite{11}
Let \(\mathcal{H} \) be a right quaternionic Hilbert space and \( A\) be a subset of \( \mathcal{H} \). Then,
\begin{enumerate}
\item $A^{\perp}=\langle A\rangle^\perp=\overline{\langle A\rangle }^\perp=\overline{\langle A\rangle ^\perp}.$
\item $(A^\perp)^\perp=\overline{\langle A\rangle}.$
\item $\overline{A}\oplus A^\perp=\mathcal{H}.$
\end{enumerate}
\end{proposition}

\begin{theorem}\cite{11}
Let \( \mathcal{H} \) be a quaternionic Hilbert space and let \( N \) be a subset of \( \mathcal{H} \) such that, for \( z, z' \in N \), we have \( \langle z \mid z' \rangle = 0 \) if \( z \neq z' \) and \( \langle z \mid z \rangle = 1 \). Then, the following conditions are equivalent:

\begin{itemize}
    \item[(a)] For every \( u, v \in \mathcal{H} \), the series \( \sum_{z \in N} \langle u \mid z \rangle \langle z \mid v \rangle \) converges absolutely and
    \[
    \langle u \mid v \rangle = \sum_{z \in N} \langle u \mid z \rangle \langle z \mid v \rangle;
    \]
    
    \item[(b)] For every \( u \in \mathcal{H} \), \( \|u\|^2 = \displaystyle{\sum_{z \in N} |\langle z \mid u \rangle|^2 }\);
    
    \item[(c)] \( N^{\perp} = \{0\} \);
    
    \item[(d)] \( \langle N \rangle \) is dense in \( \mathcal{H} \).
\end{itemize}
\end{theorem}
\begin{definition}\cite{11}
A subset $N$ of $\mathcal{H}$ that satisfies one of the statements in Theorem $2$ is called Hilbert basis or orthonormal basis for $\mathcal{H}$.
\end{definition}
\begin{theorem}\cite{11}
Every quaternionic Hilbert space has a Hilbert basis.\\
\end{theorem}

\begin{definition}[Frames]\cite{15}
Let $\{u_i\}_{i\in I}$ be a sequence in a right quaternionic Hilbert space. $\{u_i\}_{i\in I}$ is said to be Frame for $\mathcal{H}$ if there exist $0<A\leq B<\infty$ such that for all $u\in \mathcal{H}$, the following inequality holds:
$$A\|u\|^2\leq \displaystyle{\sum_{i\in I}\vert \langle u_i,u\rangle \vert^2}\leq B\|u\|^2.$$
\begin{enumerate}
\item If only the upper inequality holds, $\{u_i\}_{i\in I}$ is called a Bessel sequence for $\mathcal{H}$.
\item If $A=B=1$, $\{u_i\}_{i\in I}$ is called a Parseval frame for $\mathcal{H}$.
\end{enumerate}
\end{definition}
\begin{definition}[Right $\mathbb{H}$-linear operator]\cite{11}
Let $\mathcal{H}$ and $\mathcal{K}$ be two right quaternionic Hilbert spaces. Let  $L:\mathcal{H}\rightarrow \mathcal{K}$ be a map.
\begin{enumerate}
\item  $L$ is said to be right $\mathbb{H}$-linear operator if $L(uq+vp)=L(u)q+L(v)p$ for all $u,v\in \mathcal{H}$ and $p,q\in \mathcal{H}$.
\item  If $L$ is  a right $\mathbb{H}$-linear operator. $L$ is continuous if and only if $L$ is bounded; i.e., there exists $M> 0$ such that for all $u\in \mathcal{H}$, 
$$\| Lu\|\leq M\| u\|.$$
We denote $\mathbb{B}(\mathcal{H},\mathcal{K})$ the set of all right $\mathbb{H}$-linear bounded operators from $\mathcal{H}$ to $\mathcal{K}$, and if $\mathcal{H}=\mathcal{K}$, we denote $\mathbb{B}(\mathcal{H})$ instead of $\mathbb{B}(\mathcal{H},\mathcal{H})$.
\item If $L$ is a right $\mathbb{H}$-linear  bounded operator, we define the norm of $L$ as: 
$$\|L\|=\displaystyle{\sup_{\|u\|=1}\|Lu\|}=\inf\{M>0:\; \| Lu\|\leq M\| u\|,\; \forall u\in \mathcal{H}\}.$$
And we have for all $L,M\in \mathbb{B}(\mathcal{H}), \|L+M\|\leq\|L\|+\|M\|$ and $\|MN\|\leq\|L\|\|M\|.$\\
\end{enumerate}
\end{definition}
\begin{definition}[the adjoint operator]\cite{11}
Let $\mathcal{H}$ be a right quaternionic Hilbert space and  $L\in \mathbb{B}(\mathcal{H})$. The adjoint operator of $L$, denoted $L^*$, is the unique operator in $\mathbb{B}(\mathcal{H})$ satisfying for all $u,v\in \mathcal{H}$:
$$\langle Lu\mid v\rangle=\langle u\mid L^*v\rangle.$$
\end{definition}
\begin{theorem}[ The closed graph theorem]\label{thm7} \cite{11}
Let $\mathcal{H}, \mathcal{K}$ be two right quaternionic Hilbert spaces and let $L:\mathcal{H}\rightarrow \mathcal{K}$ be a right $\mathbb{H}$-linear opeartor. If $Graph(L)$ is closed, then $L\in \mathbb{B}(\mathcal{H},\mathcal{K})$.\\
\end{theorem}
\begin{definition}[$K$-Frames]\cite{9}
Let $K\in \mathbb{B}(\mathcal{H})$ and $\{u_i\}_{i\in I}$ be a sequence in a right quaternionic Hilbert space. $\{u_i\}_{i\in I}$ is said to be $K$-frame for $\mathcal{H}$ if there exist $0<A\leq B<\infty$ such that for all $u\in \mathcal{H}$, the following inequality holds:
$$A\|K^*u\|^2\leq \displaystyle{\sum_{i\in I}\vert \langle u_i,u\rangle \vert^2}\leq B\|u\|^2.$$
\begin{enumerate}
\item If only the upper inequality holds, $\{u_i\}_{i\in I}$ is called a Bessel sequence for $\mathcal{H}$.
\item If $A=B=1$, $\{u_i\}_{i\in I}$ is called a Parseval $K$-frame for $\mathcal{H}$.
\end{enumerate}
\end{definition}
\begin{definition}\cite{15}
Let $\{u_i\}_{i\in I}$ be a Bessel sequence for $\mathcal{H}$. 
\begin{enumerate}
\item The pre-frame operator of $\{u_i\}_{i\in I}$ is the right $\mathbb{H}$-linear bounded operator denoted by $T$ and  defined as follows: $$\begin{array}{rcl}
T:\ell^2(\mathbb{H})&\rightarrow& \mathcal{H}\\
q:=\{q_i\}_{i\in I}&\mapsto& \displaystyle{\sum_{i\in I}u_iq_i}. 
\end{array}$$
\item The transform opeartor of $\{u_i\}_{i\in I}$, denoted by $\theta$, is the adjoint of its pre-frame operator. explicitly $\theta$ is defined as follows:
$$\begin{array}{rcl}
\theta: \mathcal{H}&\rightarrow& \ell^2(\mathbb{H})\\
u&\mapsto&\{\langle u_i,u\rangle\}_{i\in I}.
\end{array}$$
\item The frame operator of $\{u_i\}_{i\in I}$, denoted by $S$, is the composite of $T$ and $\theta$. explicitly, $S$ is defined as follows: 
$$\begin{array}{rcl}
S:\mathcal{H}&\rightarrow& \mathcal{H}\\
u&\mapsto&\displaystyle{\sum_{i\in I}u_i\langle u_i,u\rangle}.
\end{array}$$
\end{enumerate}
\end{definition}
\begin{remark}
Unlike to ordinary frames, the pre-frame operator of a $K$-frame is not surjective, the transform opeartor is not injective with closed range and the frame operator is not invertible.
\end{remark}
\section{$K$-frames for quaternionic Hilbert spaces}
In this section, let $\mathcal{H}$ be a right quaternionic Hilbert space. We will characterize $K$-frames for $\mathcal{H}$ by their associated operators and provide some general results on $K$-frames in $\mathcal{H}$.

First, we present the quaternionic version of Douglas's theorem. The proof does not differ significantly from the complex case, but this does not preclude its presentation.
\begin{theorem}[Quaternionic version of Douglas's Theorem]\label{thm1}
Let $\mathcal{H}_1,\mathcal{H}_2$ and $\mathcal{H}$ be right quaternionic Hilbert spaces and let $L\in \mathbb{B}(\mathcal{H}_1,\mathcal{H})$ and $M\in \mathbb{B}(\mathcal{H}_2,\mathcal{H})$. Then, the following statements are equivalent:
\begin{enumerate}
\item $R(L)\subset R(M)$.
\item There exists a constant $c>0$ such that $LL^*\leq MM^*c$.
\item There exists $X\in \mathbb{B}(\mathcal{H}_1,\mathcal{H}_2)$ such that $L=MX$.
\end{enumerate}
\end{theorem}
\begin{proof}
Assume $3.$, i.e., there exists $X\in \mathbb{B}(\mathcal{H}_1,\mathcal{H}_2)$ such that $L=MX$. Then,  $LL^*=MXX^*M^*$. Thus, using the Cauchy-Schwarz inequality, for all $u\in \mathcal{H}$, we have: $$\langle LL^*u,u\rangle =\langle MXX^*M^*u,u\rangle \leq \|X\|^2\langle MM^*u,u\rangle.$$ Hence, $3.$ implies $2..$ It is clear that $3.$ implies $1.$. Assume $1.$, i.e., $R(L)\subset R(M)$. We can define a right $\mathbb{H}$-linear operator $X$ from $\mathcal{H}_1$ to $\mathcal{H}_2$ as follows: For $u\in \mathcal{H}_1$, $Lu\in R(L)\subset R(M)$, then there exists a unique $v\in \text{ker}(M)^\perp$ such that $Bv=Au$. Set $Xu=v$, hence $L=MX$. It remains to prove that $X$ is bounded. Let $\{u_n\}_{n\geqslant 1}\subset \mathcal{H}_1$ such that $u_n \to u$ and $Xu_n \to v$, then $Lu_n\to Lu$ and $MXu_n\to Mv$, and since $L=MX$, then $Lu_n\to Mv$. Hence, $Mv=Lu$, and since $\text{ker}(M)^\perp$ is closed, then $v\in \text{ker}(M)^\perp$, thus $Xu=v$. Hence, by the closed graph theorem, $X$ is bounded. Hence, $1.$ implies $3.$. Assume $2.$, i.e., there exists $c> 0$, such that $LL^*\leq MM^*c$. Define $D:R(M^*)\rightarrow R(L^*)$ by $D(M^*u)=L^*u$. Then $D$ is well defined and bounded since: 
\begin{equation}
\|DM^*u\|^2=\|L^*u\|^2=\langle LL^*u,u\rangle \leq c\langle MM^*u,u\rangle=c\| M^*u\|^2. 
\end{equation}
Hence, $D$ can be uniquely extended to $\overline{R(M^*)}$, and if we define $D$ on $R(M^*)^\perp$ to be zero, then $DM^*=L^*$, Hence, $L=MD^*$. Hence, $2.$ implies $3.$.
\end{proof}

The following theorem characterizes $K$-frames for a right quaternionic Hilbert spaces using the associated operators.
\begin{theorem}\label{thm2}
Let $\{u_i\}_{i\in I }$ be a Bessel sequence for $\mathcal{H}$ and $K\in \mathbb{B}(\mathcal{H})$. Then the following statements are equivalent.
\begin{enumerate}
\item $\{u_i\}_{i\in I}$ is a $K$-frame for $\mathcal{H}$.
\item $R(K)\subset R(T).$
\item There exists a constant $c>0$ such that $KK^*c\leq S$.
\item There exists $X\in \mathbb{B}(\mathcal{H},\ell^2(\mathbb{H})\,)$ such that $K=TX$.
\end{enumerate}
\end{theorem}
\begin{proof}
Since $\{u_i\}_{i\in I}$ is a Bessel sequence, then $\{u_i\}_{i\in I}$ is a $K$-frame for $\mathcal{H}$, if and only if, there exist $c> 0$ such that for all $u\in \mathcal{H}$, $c\|K^*u\|^2\leq \displaystyle{\sum_{i\in I}\vert \langle u_i,u\rangle \vert^2}.$ Since $\displaystyle{\sum_{i\in I}\vert \langle u_i,u\rangle \vert^2}=\langle Su,u\rangle$, then $\{u_i\}_{i\in I}$ is a $K$-frame for $\mathcal{H}$, if and only if, there exist $c> 0$ such that  $KK^*c\leq S$, if and only if, $R(K)\subset R(T)$ by Theorem \ref{thm1}, if and only if there exists $X\in \mathbb{B}(\mathcal{H},\ell^2(\mathbb{H})\,)$ such that $K=TX$ (by Theorem \ref{thm1}).
\end{proof}

Given a right $\mathbb{H}$-linear bounded operator and a $K$-frame, $\{u_i\}_{i\in I}$, for $\mathcal{H}$. The following proposition demonstrates the existence of a particular Bessel sequence called a $K$-dual frame to $\{u_i\}_{i\in I}$, which, together with the $K$-frame, allows for the reconstruction of all elements of $R(K)$.
\begin{proposition}\label{prop1}\cite{9}
Let $\{u_i\}_{i\in I}$ be a Bessel sequence for $\mathcal{H}$. $\{u_i\}_{i\in I}$ is  $K$-frame for $\mathcal{H}$ if and only if  there exists a Bessel sequence $\{v_i\}_{i\in I}$ for $\mathcal{H}$ such that for all $u\in \mathcal{H}$, $Ku=\displaystyle{\sum_{i\in I}u_i\langle v_i,u\rangle}$.
Such a Bessel sequence is called a $K$-dual frame to $\{u_i\}_{i\in I}$.
\end{proposition}
\begin{remark}
Let $\{u_i\}_{i\in I}$ be a $K$-frame for $\mathcal{H}$ and $\{v_i\}_{i\in I}$ be a $K$-dual frame to $\{u_i\}_{i\in I}$. Then:
\begin{enumerate}
\item For all $u\in \mathcal{H}$, $K^*u=\displaystyle{\sum_{i\in I}v_i\langle u_i,u\rangle}$. i.e., $\{v_i\}_{i\in I}$ is a $K^*$-frame for $\mathcal{H}$ with  $\{u_i\}_{i\in I}$ as a $K^*$-dual frame.
\item $\{u_i\}_{i\in I}$ and $\{v_i\}_{i\in I}$ are interchangeable if and only if $K$ is self-adjoint.
\end{enumerate}
\end{remark}

\begin{proposition}\label{prop2}
Let $\{u_i\}_{i\in I}$ be a frame for $\mathcal{H}$ and $K\in \mathbb{B}(\mathcal{H})$. Then, $\{Ku_i\}_{i\in I}$ is a $K$-frame for $\mathcal{H}$.
\end{proposition}
\begin{proof}
Since $\{u_i \}_{i\in I}$ is a frame for $\mathcal{H}$, then for all $u\in \mathcal{H}$, $u=\displaystyle{\sum_{i\in I}u_i\langle S^{-1}u_i,u\rangle}$, where $S$ is the frame operator of $\{u_i\}_{i\in I}$. Then, for all $u\in \mathcal{H}$, $Ku=\displaystyle{\sum_{i\in I}Ku_i\langle S^{-1}u_i,u\rangle}$. Hence, Proposition \ref{prop1} completes the proof.
\end{proof}

Given a sequence $\{u_i\}_{i \in I}$ in $\mathcal{H}$, the following theorem characterizes all right $\mathbb{H}$-linear bounded operators $K$ for which $\{u_i\}_{i \in I}$ is a $K$-frame and provides a $K$-dual frame for $\{u_i\}_{i\in I}$.
\begin{theorem}\label{thm3}
Let $\{u_i\}_{i\in I}$ be a Bessel sequence with the pre-frame operator $T$. Then, for all $X\in \mathbb{B}(\mathcal{H},\ell^2(\mathbb{H})\,)$, $\{u_i\}_{i\in I}$ is a $TX$-frame and $\{X^*e_i\}_{i\in I}$ is a $TX$-dual frame to $\{u_i\}_{i\in I}$, where $\{e_i\}_{i\in I}$ is the standard Hilbert basis for $\ell^2(\mathbb{H})$. Conversely, if $\{u_i\}_{i\in I}$ is a $K$-frame for $\mathcal{H}$, where $K\in \mathbb{B}(\mathcal{H})$, and $\{v_i\}_{i\in I}$ is a $K$-dual frame to $\{u_i\}_{i\in I}$, then $K=TX$, where $X$ is the transform operator of $\{v_i\}_{i\in I}$.
\end{theorem}
\begin{proof}
Let $X\in \mathbb{B}(\mathcal{H},\ell^2(\mathbb{H})\,)$. Since $R(TX)\subset R(T)$, then, by Theorem \ref{thm2}, $\{u_i\}_{i\in I}$ is a $TX$-frame for $\mathcal{H}$. We have for all $u\in \mathcal{H}$, $TXu=\displaystyle{\sum_{i\in I}u_i\langle e_i,Xu\rangle=\sum_{i\in I}u_i\langle X^*e_i,u\rangle}.$ Hence, $\{X^*e_i\}_{i\in I}$ is a $TX$-dual frame to $\{u_i\}_{i\in I}$. Conversely, assume that $\{u_i\}_{i\in I}$ is a $K$-frame for $\mathcal{H}$, where $K\in \mathbb{B}(\mathcal{H})$, and $\{v_i\}_{i\in I}$ is $K$-dual frame to $\{u_i\}_{i\in I}$. Then, for all $u\in \mathcal{H}$, $Ku=\displaystyle{\sum_{i\in I}u_i\langle v_i,u\rangle}=TXu$, where $X$ is the transform operator of $\{v_i\}_{i\in I}$.
\end{proof}
\begin{corollary}
Let $K\in \mathbb{B}(\mathcal{H})$ and $\{u_i\}_{i\in I}$ be a $K$-frame for $\mathcal{H}$. Then, for all $X\in \mathbb{B}(\mathcal{H})$, $\{u_i\}_{i\in I}$ is a $KX$-frame for $\mathcal{H}$. Moreover, if $ \{v_i\}_{i\in I}$ is $K$-dual frame to $\{u_i\}_{i\in I}$, then $\{X^*v_i\}_{i\in I}$ is a $KX$-dual frame to $\{u_i\}_{i\in I}$.
\end{corollary}
\begin{proof}
Denote by $\{e_i\}_{i\in I}$ the standard Hilbert basis for $\ell^2(\mathbb{H})$ and  $T$ the pre-frame operator of $\{u_i\}_{i\in I}$. Since $\{u_i\}_{i\in I}$ is a $K$-frame, then $R(K)\subset R(T)$. And  since $R(KX)\subset R(K)$, then $R(KX)\subset R(T)$. Hence, $\{u_i\}_{i\in I}$ is a $KX$-frame for $\mathcal{H}$. On the other hand, we have $K=T\theta$, where $\theta$ is the transform operator of $\{v_i\}_{i\in I}$. Then $KX=T\theta X$, thus $\{(\theta X)^*e_i\}_{i\in I}$ is a $KX$-dual frame to $\{u_i\}_{i\in I}$. We have for all $i\in I$, $(\theta X)^*e_i=X^*\theta^*e_i=X^*v_i$.
\end{proof}
\begin{theorem}\label{thm4}
Let \( \mathcal{B} \) denote the set of all Bessel sequences in \( \mathcal{H} \) indexed by a
countable set \( I \). Then the map:
$$\begin{array}{rcl}
\theta : \mathcal{B} &\rightarrow& \mathbb{B}(\mathcal{H}, \ell^2(\mathbb{H}))\\
\{u_i\}_{i \in I} &\mapsto& \theta_u,
\end{array}$$
where \( \theta_u \) is the transform operator of \( \{u_i\}_{i \in I} \), is bijective. Moreover,
$$\begin{array}{rcl}
\theta^{-1} : \mathbb{B}(\mathcal{H}, \ell^2(\mathbb{H})) &\rightarrow& \mathcal{B}\\
L &\mapsto& \{L^*(v_i)\}_{i \in I},
\end{array}$$
where \( \{v_i\}_{i \in I} \) is the standard Hilbert basis for \( \ell^2(\mathbb{H}) \).
\end{theorem}
\begin{proof}
Let $u=\{u_i\}_{i\in I}$, $w=\{w_i\}_{i\in I}$ be two Bessel sequences for $\mathcal{H}$. We have:
$$\begin{array}{rcl}
\theta_u=\theta_w &\Longrightarrow& \; \forall x\in \mathcal{H},\;\forall i \in I,\; \langle u_i,x\rangle=\langle w_i,x\rangle\\
&\Longrightarrow& \forall i\in I,\; u_i=w_i\\
&\Longrightarrow& u=w.
\end{array}$$
Then, $\theta$ is injective. Let $L\in \mathbb{B}(\mathcal{H}, \ell^2(\mathbb{H}))$, and set $f=\{L^*v_i\}_{i\in I}$. We have for all $u\in \mathcal{H}$, $\theta_f u=\{\langle L^*v_i,u\rangle \}_{i\in I}=\{\langle v_i,Lu\rangle \}_{i\in I}=Lu$ since $\{v_i\}_{i\in I}$ is the standard Hilbert basis for $\ell^2(\mathbb{H})$. Hence, $\theta$ is surjective, thus is bijective. Moreover: 
$$\begin{array}{rcl}
\theta^{-1} : \mathbb{B}(\mathcal{H}, \ell^2(\mathbb{H})) &\rightarrow& \mathcal{B}\\
L &\mapsto& \{L^*(v_i)\}_{i \in I},
\end{array}$$
where \( \{v_i\}_{i \in I} \) is the standard Hilbert basis for \( \ell^2(\mathbb{H}) \).
\end{proof}
\begin{remark}
The map $\theta$, defined in Theorem \ref{thm8}, is an invertible $\mathbb{R}$-linear operator.\\
\end{remark}

Theorems \ref{thm3} and \ref{thm4} lead to the following proposition.
\begin{proposition}\label{prop3}
Let $K\in \mathbb{B}(\mathcal{H})$, $\{u_i\}_{i\in I}$ be a $K$-frame for $\mathcal{H}$ and $T$ be its pre-frame operator. Then:
$$\operatorname{card}\left\{\,\{v_i\}_{i\in I} \text{ Bessel sequence $K$-dual to } \{u_i\}_{i\in I}\right \}=\operatorname{card}\{ X\in \mathbb{B}(\mathcal{H},\ell^2(\mathbb{H})): K=TX\}.$$
\end{proposition}
\begin{definition}
Let $K\in \mathbb{B}(\mathcal{H})$. a $K$-frame for $\mathcal{H}$ is said to be $K$-minimal frame if its synthesis operator is injective.
\end{definition}
\begin{remark}
A $K$-minimal frame for $\mathcal{H}$ does not contain zeros. i.e., if $\{u_i\}_{i\in I}$ is a $K$-minimal frame, then for all $i\in I$, $u_i\neq 0$.
\end{remark}

The following theorem characterizes the KK-minimal frames by duality .
\begin{theorem}
Let $K\in \mathbb{B}(\mathcal{H})$, $\{u_i\}_{i\in I}$ be a $K$-frame for $\mathcal{H}$ and $T$ be its pre-frame operator. Then the following statements are equivalent.
\begin{enumerate}
\item $\{u_i\}_{i\in I}$ has a unique $K$-dual frame.
\item There exists a unique $X\in \mathbb{B}(\mathcal{H},\ell^2(\mathbb{H}))$ such that $K=TX$.
\item $\{u_i\}_{i\in I}$ is a $K$-minimal frame for $\mathcal{H}$.
\end{enumerate}
\end{theorem}
\begin{proof}
Denote by $\{e_i\}_{i\in I}$ the standard Hilbert basis for $\ell^2(\mathbb{H})$. By Proposition \ref{prop3}, $1.\Longleftrightarrow 2.$ It is clear that, if $T$ is injective, then $X\in \mathbb{B}(\mathcal{H},\ell^2(\mathbb{H}))$ is unique. Hence, $3.$ implies $2.$ Assume $2$., then $\{u_i\}_{i \in I}$ has a unique $K$-dual $\{X^*e_i\}_{i \in I}$. We can suppose that $X^*e_i \neq 0$ for all $i \in I$. Suppose, by contradiction, that $T$ is not injective. Then there exists $x_0 \neq 0$ such that $Tx_0 = 0$, i.e.
\[
\sum_{i \in I} u_i \langle e_i, x_0 \rangle = 0.
\]
Let $j \in I$ such that $\langle e_j, x_0 \rangle \neq 0$. Then $u_j = \displaystyle{\sum_{i \neq j}  u_i\alpha_i}$ where $\alpha_i = \displaystyle{\frac{\langle e_i, x_0 \rangle}{\langle e_j, x_0 \rangle}}$ for all $i \neq j$. 

It is clear that $\{\alpha_i\}_{i \in I \setminus \{j\}} \in \ell^2(\mathbb{H})$. We have for all $x \in H$,
$$\begin{array}{rcl}
Kx = \displaystyle{\sum_{i \in I} u_i \langle X^*e_i, x \rangle} &=& \displaystyle{\sum_{i \neq j} u_i \langle X^*e_i, x \rangle + u_j \langle X^*(e_j), x \rangle}\\
& =& \displaystyle{\sum_{i \neq j} u_i \langle X^*e_i, x \rangle + \sum_{i \neq j} u_i \alpha_i \langle X^*e_j, x \rangle}.
\end{array}$$

Thus,
\[
Kx = \sum_{i \neq j} u_i \langle X^*(e_i +  e_j\overline{\alpha_i}), x \rangle.
\]

Since $\{\alpha_i\}_{i \in I \setminus \{j\}} \in \ell^2(\mathbb{H})$, we easily show that $\{g_i\}_{i \in I}$ is a Bessel sequence for $H$, where $g_i = X^*(e_i + e_j\overline{\alpha_i})$ for $i \neq j$ and $g_j = 0$. Then $\{g_i\}_{i \in I}$ is another $K$-dual frame to $\{u_i\}_{i \in I}$ than $\{X^*e_i\}_{i \in I}$ since $X^*e_j \neq 0 = g_j$. Contradiction.
\end{proof}
\begin{definition}
Let $K\in \mathbb{B}(\mathcal{H})$ and  $\{u_i\}_{i\in I}$ be a Bessel sequence for $\mathcal{H}$. $\{u_i\}_{i\in I}$ is said to be $K$-orthonormal basis for $\mathcal{H}$ if the following holds:
\begin{enumerate}
\item $\{u_i\}_{i\in I}$ is an orthonormal system in $\mathcal{H}$.
\item $\{u_i\}_{i\in I}$ is a Parseval $K$-frame for $\mathcal{H}$.
\end{enumerate}
\end{definition}
The following theorem expresses the unique $K$-dual frame to  a $K$-orthonormal basis $\{u_i\}_{i\in I}$ in terms of $K$ and $\{u_i\}_{i\in I}$.
\begin{theorem}
Let $K\in \mathbb{B}(\mathcal{H})$ and $\{u_i\}_{i\in I}$ be a $K$-orthonormal basis for $\mathcal{H}$. Then, $\{u_i\}_{i\in I}$ has a unique $K$-dual frame which is exactly $\{K^*u_i\}_{i\in I}$.
\end{theorem}
\begin{proof}
Since $\{u_i\}_{i\in i}$ is an orthonormal system, then its pre-frame operator is injective. Hence, $\{u_i\}_{i\in I}$ has a unique $K$-dual frame. Let's $\{v_i\}_{i\in I}$ be the unique $K$-dual frame to $\{u_i\}_{i\in I}$, then for all $u\in \mathcal{H}$, $Ku=\displaystyle{\sum_{i\in I}u_i\langle v_i,u\rangle}$. Let $j\in I$, then for all $u\in \mathcal{H}$, we have: 
$$\begin{array}{rcl}
\langle u_j,Ku\rangle&=&\left\langle u_j,\displaystyle{\sum_{i\in I}u_i\langle v_i,u\rangle}\right\rangle\\
&=&\displaystyle{\sum_{i\in I}\langle u_j,u_i\rangle\langle v_i,u\rangle}\\
&=& \langle v_j,u\rangle\;\; \text{ since } \{u_i\}_{i\in I} \text{ is an orthonormal system}.
\end{array}$$
Then, for all $u\in \mathcal{H}$, $\langle K^*u_j,u\rangle =\langle v_j,u\rangle$. Hence, $v_j=K^*u_j$ for all $j\in I$.
\end{proof}
\section{ $K$-frames on quaternionic super Hilbert spaces}
In this section, \( \mathcal{H}_1 \) and \( \mathcal{H}_2 \) are two right quaternionic Hilbert spaces. \( \mathcal{H}_1 \oplus \mathcal{H}_2 \) is the direct sum  \( \mathcal{H}_1 \) and \( \mathcal{H}_2 \). The space \( \mathcal{H}_1 \oplus \mathcal{H}_2 \), equipped with the inner product
\[
\langle u_1 \oplus v_1, u_2 \oplus v_2 \rangle := \langle u_1, u_2 \rangle_{\mathcal{H}_1} + \langle v_1, v_2 \rangle_{\mathcal{H}_2},
\]
for all \( u_1, u_2 \in \mathcal{H}_1 \) and \( v_1, v_2 \in \mathcal{H}_2 \), is clearly a right quaternionic Hilbert space, which we call the super right quaternionic  Hilbert space of \( \mathcal{H}_1 \) and \( \mathcal{H}_2 \). \( \mathbb{B}(\mathcal{H}_1, \mathcal{H}_2) \) denotes the collection of all right $\mathbb{H}$-linear  bounded operators from \( \mathcal{H}_1 \) to \( \mathcal{H}_2 \). For \(\mathcal{H}_1=\mathcal{H}_2=\mathcal{H}\), we denote, simply, \(\mathbb{B}(\mathcal{H})$. For \( K \in \mathbb{B}(\mathcal{H}_1) \) and \( L \in \mathbb{B}(\mathcal{H}_2) \), we denote by \( K \oplus L \) the right \(\mathbb{H}\)-linear bounded operator on \( \mathcal{H}_1 \oplus \mathcal{H}_2 \), defined for all \( u \in \mathcal{H}_1 \) and \( v\in \mathcal{H}_2 \) by:
\[
(K \oplus L)(u \oplus v) = Ku \oplus Lv.
\]

For a right $\mathbb{H}$-linear bounded operator \( L\), \( \mathcal{R}(L) \) and \( \mathcal{N}(L) \) denote the range and the kernel, respectively, of \( L\).
In what follows, without any possible confusion, all inner products are denoted by the same notation \( \langle \cdot, \cdot \rangle \) and all norms are denoted by the same notation \( \| \cdot \| \).

First, we present this proposition that demonstrates the relationship between a quaternionic super Hilbert space and the quaternionic Hilbert spaces that constitute it.
\begin{proposition}
The map
$$\begin{array}{rcl}
P_1 : \mathcal{H}_1 \oplus \mathcal{H}_2 &\longrightarrow &\mathcal{H}_1 \oplus \mathcal{H}_2\\
u \oplus v &\mapsto& u \oplus 0,
\end{array}$$
and the map
$$\begin{array}{rcl}
P_2 : \mathcal{H}_1 \oplus \mathcal{H}_2 &\longrightarrow& \mathcal{H}_1 \oplus \mathcal{H}_2\\
u \oplus v &\mapsto& 0 \oplus v,
\end{array}$$
are two orthogonal projections on \( \mathcal{H}_1 \oplus \mathcal{H}_2 \) with \( R(P_1) = \mathcal{H}_1 \oplus 0 \) and \( R(P_2) = 0 \oplus \mathcal{H}_2 \).
\end{proposition}
\begin{proof}
We have, clearly, \( P_1^2 = P_1 \) and \( P_2^2 = P_2 \). On the other hand, for \( u \oplus v, a \oplus b \in \mathcal{H}_1 \oplus \mathcal{H}_2 \), we have:
\[
\langle P_1(u \oplus v), a \oplus b \rangle = \langle u \oplus 0, a \oplus b \rangle
= \langle u, a \rangle + \langle 0, b \rangle
= \langle u, a \rangle + \langle v, 0 \rangle
= \langle u \oplus v, a \oplus 0 \rangle.
\]
Hence, for all \( a \oplus b \in \mathcal{H}_1 \oplus \mathcal{H}_2 \), \( P_1^*(a \oplus b) = a \oplus 0 = P(a \oplus b) \). Then \( P_1^* = P_1 \).

Similarly, we show that \( P_2^* = P_2 \). Thus, \( P_1 \) and \( P_2 \) are orthogonal projections.

It is clear that \( R(P_1) = \{ u \oplus 0 \,|\, u \in \mathcal{H}_1 \} := \mathcal{H}_1 \oplus 0 \) and \( R(P_2) = \{ 0 \oplus v \,|\, v \in \mathcal{H}_2 \} := 0 \oplus \mathcal{H}_2 \).
\end{proof}
The following proposition states that a sequence in $\mathcal{H}_1\oplus \mathcal{H}_2$ is a Bessel sequence if and only if each component is a Bessel sequence for the associated space.
\begin{proposition}\label{prop6}
Let $\{u_i\}_{i \in I}$ and $\{v_i\}_{i \in  I}$ be two sequences in $\mathcal{H}_1$ and $\mathcal{H}_2$ respectively. The following statements are equivalent:
\begin{enumerate}
    \item $\{u_i \oplus v_i\}_{i \in I}$ is a Bessel sequence for $\mathcal{H}_1 \oplus \mathcal{H}_2$.
    \item $\{u_i\}_{i \in I}$ and $\{v_i\}_{i \in I}$ are two Bessel sequences for $\mathcal{H}_1$ and $\mathcal{H}_2$ respectively.
\end{enumerate}
\end{proposition}
\begin{proof}

Assume that \(\{u_i \oplus v_i\}_{i \in I}\) is a Bessel sequence in \(\mathcal{H}_1 \oplus \mathcal{H}_2\) with Bessel bound \(B\). Then \(\{P_1(u_i \oplus v_i)\}_{i \in I} = \{u_i \oplus 0\}_{i \in I}\) is a Bessel sequence for \(\mathcal{H}_1 \oplus 0\) and \(B\) is a Bessel bound. That means that for all \(u \oplus 0 \in \mathcal{H}_1 \oplus 0\),

\[
\sum_{i\in I} |\langle  u_i \oplus 0,u \oplus 0 \rangle|^2 \leq B \|u \oplus 0\|^2.
\]

Hence, for all \(u \in \mathcal{H}_1\),

\[
\sum_{i\in I} |\langle u_i,u \rangle|^2 \leq B \|u\|^2.
\]

Then \(\{u_i\}_{i \in I}\) is a Bessel sequence in \(\mathcal{H}_1\). By taking \(P_2\) instead of \(P_1\), we prove, similarly, that \(\{v_i\}_{i \in I}\) is a Bessel sequence for \(\mathcal{H}_2\).

Conversely, assume that \(\{u_i\}_{i \in I}\) and \(\{v_i\}_{i \in I}\) are two Bessel sequences in \(\mathcal{H}_1\) and \(\mathcal{H}_2\) respectively with Bessel bounds \(B_1\) and \(B_2\) respectively. Note \(B = 2\max\{B_1, B_2\}\). Then for all \(u \oplus v \in \mathcal{H}_1 \oplus \mathcal{H}_2\), we have:

$$\begin{array}{rcl}
\displaystyle{\sum_{i\in I} |\langle u_i \oplus v_i,u \oplus v \rangle|^2} &=& \displaystyle{\sum_{i\in I} |\langle u_i,u \rangle + \langle v_i,v \rangle|^2}\\
&\leq& \displaystyle{\sum_{i\in I} 2\left(|\langle  u_i,u \rangle|^2 + |\langle v_i,v \rangle|^2\right)}\\
&\leq& 2B_1 \|u\|^2 + 2B_2 \|v\|^2\\
&\leq& B(\|u\|^2 + \|v\|^2) = B\|u \oplus v\|^2.
\end{array}$$
Hence, \(\{u_i \oplus v_i\}_{i \in I}\) is a Bessel sequence for \(\mathcal{H}_1 \oplus \mathcal{H}_2\).
\end{proof}
The following proposition expresses the operators associated with a super Bessel sequence  in terms of the operators associated with its components.
\begin{proposition}\label{prop7}
Let \(\{u_i\}_{i \in I}\) and \(\{v_i\}_{i \in I}\) be two Bessel sequences for \(\mathcal{H}_1\) and \(\mathcal{H}_2\) respectively. Let \(T_1\), \(T_2\), and \(T\) be the synthesis operators of \(\{u_i\}_{i \in I}\), \(\{v_i\}_{i \in I}\), and \(\{u_i \oplus v_i\}_{i \in I}\) respectively. Let \(\theta_1\), \(\theta_2\), and \(\theta\) be the frame transforms of \(\{u_i\}_{i \in I}\), \(\{v_i\}_{i \in I}\), and \(\{u_i \oplus v_i\}_{i \in I}\) respectively. Let \(S_1\), \(S_2\), and \(S\) be the frame operators of \(\{u_i\}_{i \in I}\), \(\{v_i\}_{i \in I}\), and \(\{u_i \oplus v_i\}_{i \in I}\) respectively. Then:

\begin{enumerate}
    \item For all \(a \in \ell^2(\mathbb{H})\), \(T(a) = T_1(a) \oplus T_2(a)\).
    \item For all \(u \oplus v \in \mathcal{H}_1 \oplus \mathcal{H}_2\), \(\theta(u \oplus v) = \theta_1(u) + \theta_2(v)\).
    \item For all \(u \oplus v \in \mathcal{H}_1 \oplus \mathcal{H}_2\), \(S(u \oplus v) = S_1(u) + T_1\theta_2(v) \oplus S_2(v) + T_2\theta_1(u)\).
\end{enumerate}
\end{proposition}
\begin{proof}
\begin{enumerate}

\item Let \( q = \{q_i\}_{i \in I} \in \ell_2(\mathbb{H}) \), we have:
\[
T(a) = \sum_{i=1}^{\infty}  (u_i \oplus v_i) q_i = \sum_{i=1}^{\infty} u_i q_i \oplus \sum_{i=1}^{\infty}  v_i q_i = T_1(a) \oplus T_2(a).
\]

\item  Let \( u \oplus v \in \mathcal{H}_1 \oplus \mathcal{H}_2 \), we have:
\[
\theta(u \oplus v) = \{ \langle  u_i \oplus v_i,u \oplus v \rangle \}_{i \in I} = \{ \langle  u_i,u \rangle + \langle  v_i,v \rangle \}_{i \in I} = \{ \langle  u_i,u \rangle \}_{i \in I} + \{ \langle v_i,v \rangle \}_{i \in I} = \theta_1(u) + \theta_2(v).
\]

\item  Let \( u \oplus v \in \mathcal{H}_1 \oplus \mathcal{H}_2 \), we have:
$$\begin{array}{rcl}
S(u \oplus v) = T\theta(u \oplus v) = T(\theta_1(u) + \theta_2(v)) &=& T(\theta_1(u)) + T(\theta_2(v))\\
& =& T_1(\theta_1(u)) \oplus T_2(\theta_1(u)) + T_1(\theta_2(v)) \oplus T_2(\theta_2(v))\\
&=& S_1(u) \oplus T_2\theta_1(u) + T_1\theta_2(v) \oplus S_2(v)\\
&=&S_1(u) + T_1\theta_2(v) \oplus S_2(v) + T_2\theta_1(u).
\end{array}$$
\end{enumerate}
\end{proof}
The following theorem presents a necessary condition for a sequence in $\mathcal{H}_1\oplus \mathcal{H}_2$ to be a $K$-frame, where $K\in \mathbb{B}(\mathcal{H}_1\oplus \mathcal{H}_2)$.
\begin{theorem}
Let \( K \in \mathbb{B}(\mathcal{H}_1 \oplus \mathcal{H}_2) \) and \( \{u_i \oplus v_i\}_{i \in I} \) be a sequence in \( \mathcal{H}_1 \oplus \mathcal{H}_2 \). If \( \{u_i \oplus v_i\}_{i \in I} \) is a \( K \)-frame for \( \mathcal{H}_1 \oplus \mathcal{H}_2 \) with \( K \)-frame bounds \( A \) and \( B \), then:

i. For all \( u \in \mathcal{H}_1 \),
\[
A \|K^*_1(u)\|^2 \leq \sum_{i \in I} |\langle u_i, u \rangle|^2 \leq B \|u\|^2.
\]

ii. For all \( v \in \mathcal{H}_2 \),
\[
A \|K^*_2(v)\|^2 \leq \sum_{i \in I} |\langle v_i, v \rangle|^2 \leq B \|v\|^2.
\]

Where \( K_1 : \mathcal{H}_1 \oplus \mathcal{H}_2 \rightarrow \mathcal{H}_1 \) and \( K_2 : \mathcal{H}_1 \oplus \mathcal{H}_2 \rightarrow \mathcal{H}_2 \) are the right $\mathbb{H}$-linear  bounded operators such that for all \( u \oplus v \in \mathcal{H}_1 \oplus \mathcal{H}_2 \),
\[
K(u \oplus v) = K_1(u \oplus v) \oplus K_2(u \oplus v).
\]
\end{theorem}
\begin{proof}
We have for all \( u \oplus v \in \mathcal{H}_1 \oplus \mathcal{H}_2 \),
\[
A \| K^*(u \oplus v) \|^2 \leq 
\sum_{i \in I} 
| \langle u_i \oplus v_i, u \oplus v \rangle |^2 \leq B \| u \oplus v \|^2.
\]
i. In particular, for \( v = 0 \), we have for all \( u \in \mathcal{H}_1 \),
\[
A \| K^*(u \oplus 0) \|^2 \leq 
\sum_{i \in I} 
| \langle u_i, u \rangle |^2 \leq B \| u \|^2.
\]
Let \( a \oplus b \in \mathcal{H}_1 \oplus \mathcal{H}_2 \) and let \( u \in \mathcal{H}_1 \), we have:
\[
\langle K(a \oplus b), u \oplus 0 \rangle = \langle K_1(a \oplus b) \oplus K_2(a \oplus b), u \oplus 0 \rangle 
= \langle K_1(a \oplus b), u \rangle 
= \langle a \oplus b, K_1^*(u) \rangle.
\]
Hence \( K^*(u \oplus 0) = K_1^*(u) \). Then for all \( u \in \mathcal{H}_1 \),
\[
A \| K_1^*(u) \|^2 \leq 
\sum_{i \in I} 
| \langle u_i, u \rangle |^2 \leq B \| u \|^2.
\]
ii. In particular, for \( u = 0 \), we have for all \( v \in \mathcal{H}_2 \),
\[
A \| K^*(0 \oplus v) \|^2 \leq 
\sum_{i \in I} 
| \langle v_i, v \rangle |^2 \leq B \| v \|^2.
\]
Let \( a \oplus b \in \mathcal{H}_1 \oplus \mathcal{H}_2 \) and let \( v \in \mathcal{H}_2 \), we have:
\[
\langle K(a \oplus b), 0 \oplus v \rangle = \langle K_1(a \oplus b) \oplus K_2(a \oplus b), 0 \oplus v \rangle 
= \langle K_2(a \oplus b), v \rangle 
= \langle a \oplus b, K_2^*(v) \rangle.
\]
Hence \( K^*(0 \oplus v) = K_2^*(v) \). Then for all \( v \in \mathcal{H}_2 \),
\[
A \| K_2^*(v) \|^2 \leq 
\sum_{i \in I} 
| \langle v_i, v \rangle |^2 \leq B \| v \|^2.
\]
\(\square\)
\end{proof}
The following result shows that a $K_1\oplus K_2$-frame for $\mathcal{H}_1\oplus\mathcal{H}_2$ is necessarily the sum of a $K_1$-frame for $\mathcal{H}_1$ and a $K_2$-frame for $\mathcal{H}_2$.
\begin{corollary}\label{cor2}
Let \( K_1 \in \mathbb{B}(\mathcal{H}_1) \) and \( K_2 \in \mathbb{B}(\mathcal{H}_2) \). If a sequence \( \{ u_i \oplus v_i \}_{i \in I} \) in \( \mathcal{H}_1 \oplus \mathcal{H}_2 \) is a \( K_1 \oplus K_2 \)-frame for \( \mathcal{H}_1 \oplus \mathcal{H}_2 \), then \( \{ u_i \}_{i \in I} \) is a \( K_1 \)-frame for \( \mathcal{H}_1 \) and \( \{ v_i \}_{i \in I} \) is an \( K_2 \)-frame for \( \mathcal{H}_2 \).
\end{corollary}
\begin{corollary}
Let \( K_1 \in \mathbb{B}(\mathcal{H}_1) \) and \( K_2 \in \mathbb{B}(\mathcal{H}_2) \). Then there exist a \( K_1 \)-frame \( \{ u_i \}_{i \in I} \) for \( \mathcal{H}_1 \) and a \( K_2 \)-frame \( \{ v_i \}_{i \in I} \) for \( \mathcal{H}_2 \) such that \( \{ u_i \oplus v_i \}_{i \in I} \) is a \( K_1 \oplus K_2 \)-frame for \( \mathcal{H}_1 \oplus \mathcal{H}_2 \).
\end{corollary}
\begin{proof}
Take any frame \( \{ u_i \oplus v_i \}_{i \in I} \) for \( \mathcal{H}_1 \oplus \mathcal{H}_2 \). By Proposition \ref{prop2}, \( \{ K_1 \oplus K_2 (u_i \oplus v_i) \}_{i \in I} \) is a \( K_1 \oplus K_2 \)-frame for \( \mathcal{H}_1 \oplus \mathcal{H}_2 \). That means that \( \{ K_1(u_i) \oplus K_2(v_i) \}_{i \in I} \) is a \( K_1 \oplus K_2 \)-frame for \( \mathcal{H}_1 \oplus \mathcal{H}_2 \). Then, by Corollary \ref{cor2}, \( \{ K_1(u_i) \}_{i \in I} \) is a \( K_1 \)-frame for \( \mathcal{H}_1 \) and \( \{ K_2(v_i) \}_{i \in I} \) is a \( K_2 \)-frame for \( \mathcal{H}_2 \).
\end{proof}
\begin{lemma}
Let \( K_1 \in \mathbb{B}(\mathcal{H}_1) \) and \( K_2 \in \mathbb{B}(\mathcal{H}_2) \). Then \( (K_1 \oplus K_2)^* = K_1^* \oplus K_2^* \).
\end{lemma}
\begin{proof}
Let \( u \oplus v, a \oplus b \in \mathcal{H}_1 \oplus \mathcal{H}_2 \), we have:
\[
\begin{array}{rcl}
\langle K_1 \oplus K_2 (u \oplus v), a \oplus b \rangle & = & \langle K_1(u) \oplus K_2(v), a \oplus b \rangle \\
& = & \langle K_1(u), a \rangle + \langle K_2(v), b \rangle \\
& = & \langle u, K_1^*(a) \rangle + \langle v, K_2^*(b) \rangle \\
& = & \langle u \oplus v, K_1^*(a) \oplus K_2^*(b) \rangle \\
& = & \langle u \oplus v, (K_1^* \oplus K_2^*)(a \oplus b) \rangle.
\end{array}
\]
\(\square\)

\end{proof}
The following proposition shows that there is no \( K_1 \oplus K_2 \)-frame for \( \mathcal{H} \oplus \mathcal{H} \) in the form \( \{ u_i \oplus u_i \}_{i \in I} \) whenever \( K_1, K_2 \neq 0 \).
\begin{proposition}\label{prop8}
Let \( K_1, K_2 \in \mathbb{B}(\mathcal{H}) \) and \( \{ u_i \oplus v_i \}_{i \in I} \) be a Bessel sequence for \( \mathcal{H} \oplus \mathcal{H} \).
Then:
\[
\{ u_i \oplus u_i \}_{i \in I} \text{ is a } K_1 \oplus K_2\text{-frame for } \mathcal{H} \oplus \mathcal{H} \iff K_1 = K_2 = 0.
\]
\end{proposition}
\begin{proof}
Assume that \( K_1 \neq 0 \) or \( K_2 \neq 0 \). We can suppose the case of \( K_1 \neq 0 \) and the other case will be obtained similarly. For \( u \in \mathcal{H} \) such that \( K_1^*(u) \neq 0 \), we have 
\[
\| K_1^* \oplus K_2^*(u \oplus (-u)) \|^2 = \| K_1^*(u) \|^2 + \| K_2^*(u) \|^2 \geq \| K_1^*(u) \|^2 \neq 0 
\]
and 
\[
\sum_{i \in I} | \langle  u_i \oplus u_i,u \oplus (-u) \rangle |^2 = 0.
\]
Hence \( \{ u_i \oplus u_i \}_{i \in I} \) is not a \( K_1 \oplus K_2 \)-frame for \( \mathcal{H}_1 \oplus \mathcal{H}_2 \). Conversely, assume that \( K_1 = K_2 = 0 \). The result follows immediately from the fact that \( \{ u_i \oplus v_i \}_{i \in I} \) is a Bessel sequence.
\end{proof}

Proposition \ref{prop8} showed that the direct sim of a $K_1$-frame and a $K_2$-frame is not necessary a $K_1\oplus K_2$-frame for the right quaternionic super Hilbert space. The following result is a sufficient condition for a sequence \( \{ u_i \oplus v_i \}_{i \in I} \) in \( \mathcal{H}_1 \oplus \mathcal{H}_2 \) to be a \( K_1 \oplus K_2 \)-frame.
\begin{theorem}\label{thm8}
Let \( K_1 \in \mathbb{B}(\mathcal{H}_1) \) and \( K_2 \in \mathbb{B}(\mathcal{H}_2) \). Let \( \{ u_i \}_{i \in I} \) be a \( K_1 \)-frame for \( \mathcal{H}_1 \) and \( \{ v_i \}_{i \in I} \) be a \( K_2 \)-frame for \( \mathcal{H}_2 \). Let \( \theta_1 \) and \( \theta_2 \) be the frame transforms of \( \{ u_i \}_{i \in I} \) and \( \{ v_i \}_{i \in I} \) respectively. If \( R(\theta_1) \perp R(\theta_2) \) in \( \ell^2(\mathbb{H}) \), then \( \{ u_i \oplus v_i \}_{i \in I} \) is a \( K_1 \oplus K_2 \)-frame for \( \mathcal{H}_1 \oplus \mathcal{H}_2 \).
\end{theorem}
\begin{proof}
Since $\{u_i\}_{i\in I}$ and $\{v_i\}_{i\in I}$ are Bessel sequences, then, by Proposition \ref{prop6}, $\{u_i\oplus v_i\}_{i\in I}$ is a Bessel sequence for $\mathcal{H}_1\oplus\mathcal{H}_2$. Denote by \( A_1 \) and \( A_2 \) the lower bounds for \( \{ u_i \}_{i \in I} \) and \( \{ v_i \}_{i \in I} \) respectively, and let \( A = \min \{ A_1, A_2 \} \). We have for all \( u \in \mathcal{H}_1, v \in \mathcal{H}_2 \), \( \langle \theta_1(u), \theta_2(v) \rangle = 0 \). Then for all \( u \in \mathcal{H}_1, v \in \mathcal{H}_2 \), \( \langle T_2 \theta_1(u), v \rangle = 0 \) and \( \langle T_1 \theta_2(v), u \rangle = 0 \), where \( T_1 \) and \( T_2 \) are the synthesis operators for \( \{ u_i \}_{i \in I} \) and \( \{ v_i \}_{i \in I} \) respectively. Hence \( T_2 \theta_1 = T_1 \theta_2 = 0 \). By Proposition \ref{prop7}, the frame operator \( S \) of \( \{ u_i \oplus v_i \}_{i \in I} \) is defined by 
\[
S(u \oplus v) = S_1(u) \oplus S_2(v) = (S_1 \oplus S_2)(u \oplus v),
\]
where \( S_1 \) and \( S_2 \) are the frame operators for \( \{ u_i \}_{i \in I} \) and \( \{ v_i \}_{i \in I} \) respectively. 

Then: 
$$\begin{array}{rcl}
\langle S(u \oplus v), u \oplus v \rangle &=& \langle S_1(u), u \rangle + \langle S_2(v), v \rangle\\ 
&\geq& A_1 \| K_1^*(u) \|^2 + A_2 \| K_2^*(v) \|^2\\ 
&\geq &A \| K_1^*(u) \oplus K_2^*(v) \|^2\\ 
&\geq& A \| (K_1 \oplus K_2)^*(u \oplus v) \|^2.
\end{array}$$

Hence, by theorem \ref{thm2}, \( \{ u_i \oplus v_i \}_{i \in I} \) is a \( K_1 \oplus K_2 \)-frame for \( \mathcal{H}_1 \oplus \mathcal{H}_2 \).
\end{proof}
The followibg theorem presents a necessary condition for a sequence in $\mathcal{H}_1\oplus \mathcal{H}_2$ to be a $K_1\oplus K_2$-frame.
\begin{theorem}
Let \( K_1 \in \mathbb{B}(\mathcal{H}_1) \) and \( K_2 \in \mathbb{B}(\mathcal{H}_2) \). 
Let \( \{u_i\}_{i \in I} \) be a \( K_1 \)-frame for \( \mathcal{H}_1 \) 
and \( \{v_i\}_{i \in I} \) be a \( K_2 \)-frame for \( \mathcal{H}_2 \). 
If \( \{u_i \oplus v_i\}_{i \in I} \) is a \( K_1 \oplus K_2 \)-frame for \( \mathcal{H}_1 \oplus \mathcal{H}_2 \), then:
\[
\begin{cases}
R(K_1) \subset T_2(N(T_1)),\\
R(K_2) \subset T_1(N(T_2)).
\end{cases}
\]
Where \( T_1 \) and \( T_2 \) are the synthesis operators for \( \{u_i\}_{i \in I} \) and \( \{v_i\}_{i \in I} \) respectively.
\end{theorem}
\begin{proof}
Assume that \( \{u_i \oplus v_i\}_{i \in I} \) is a \( K_1 \oplus K_2 \)-frame for \( \mathcal{H}_1 \oplus \mathcal{H}_2 \) and let \( T \) be its synthesis operator. By Theorem \ref{thm2}, \( R(K_1 \oplus K_2) \subset R(T) \). Then for all \( u \oplus v \in \mathcal{H}_1 \oplus \mathcal{H}_2 \), there exists \( a \in \ell^2(\mathbb{H}) \) such that \( K_1 \oplus K_2 (u \oplus v) = Ta \). 
That means, by Proposition \ref{prop7}, that for all \( u \oplus v \in \mathcal{H}_1 \oplus \mathcal{H}_2 \), there exists \( a \in \ell^2(\mathbb{H}) \) such that \( K_1(u) \oplus K_2(v) = T_1(a) \oplus T_2(a) \). By taking \( v = 0 \), there exists \( a \in \ell^2(\mathbb{H}) \) such that \( K_1(u) = T_1(a) \) and \( 0 = K_2(v) = T_2(a) \). Hence, \( K_1(u) = T_1(a) \) and \( a \in N(T_2) \). Therefore, \( R(K_1) \subset T_1(N(T_2)) \).

By taking \( u = 0 \), there exists \( a \in \ell^2(\mathbb{H}) \) such that \( 0 = K_1(u) = T_1(a) \) and \( K_2(v) = T_2(a) \). Hence, \( K_2(v) = T_2(a) \) and \( a \in N(T_1) \). Therefore, \( R(K_2) \subset T_2(N(T_1)) \).
\end{proof}

The following result shows that the non-minimality of the two sequences \( \{u_i\}_{i \in I} \) and \( \{v_i\}_{i \in I} \) is necessary for their direct sum to be a \( K_1 \oplus K_2 \)-frame whenever \( K_1, K_2 \neq 0 \).

\begin{corollary} Let \( K_1 \in B(\mathcal{H}_1) \) and \( K_2 \in B(\mathcal{H}_2) \). Let \( \{u_i \oplus v_i\}_{i \in I} \) be a \( K_1 \oplus K_2 \)-frame for \( \mathcal{H}_1 \oplus \mathcal{H}_2 \). Then:
\begin{enumerate}
    \item If \( \{u_i\}_{i \in I} \) is \( K_1 \)-minimal, then \( K_2 = 0 \).
    \item If \( \{v_i\}_{i \in I} \) is \( K_2 \)-minimal, then \( K_1 = 0 \).
    \item If \( \{u_i\}_{i \in I} \) is \( K_1 \)-minimal and \( \{v_i\}_{i \in I} \) is \( K_2 \)-minimal, then \( K_1 = 0 \) and \( K_2 = 0 \).
\end{enumerate}
In particular: Let \( K_1 \neq 0 \) and \( K_2 \neq 0 \). If at least one of \( \{u_i\}_{i \in I} \) and \( \{v_i\}_{i \in I} \) is minimal, then \( \{u_i \oplus v_i\}_{i \in I} \) is never a \( K_1 \oplus K_2 \)-frame for the right quaternionic super Hilbert space \( \mathcal{H}_1 \oplus \mathcal{H}_2 \).\\
\end{corollary}
The following result gives a necessary and sufficient condition for a $K$-frame in $\mathcal{H}_1\oplus \mathcal{H}_2$.
to be a $K$-minimal frame.
\begin{proposition}\label{prop9}
Let \( K \in B(\mathcal{H}_1 \oplus \mathcal{H}_2) \). Let \( \{u_i \oplus v_i\}_{i \in I} \) be a \( K \)-frame for \( \mathcal{H}_1 \oplus \mathcal{H}_2 \). Then the following statements are equivalent:
\begin{enumerate}
    \item \( \{u_i \oplus v_i\}_{i \in I} \) is a \( K \)-minimal frame for \( \mathcal{H}_1 \oplus \mathcal{H}_2 \).
    \item \( N(T_1) \cap N(T_2) = \{0\} \).
\end{enumerate}
where \( T_1 \) and \( T_2 \) are the synthesis operators of \( \{u_i\}_{i \in I} \) and \( \{v_i\}_{i \in I} \) respectively.
\end{proposition}
\begin{proof}
By Proposition \ref{prop7}, for all \( a \in \ell^2(\mathbb{H}) \), we have \( T(a) = T_1(a) \oplus T_2(a) \), where \( T \) is the synthesis operator of \( \{u_i \oplus v_i\}_{i \in I} \). Then \( N(T) = N(T_1) \cap N(T_2) \). Hence:

\[
\{u_i \oplus v_i\}_{i \in I} \text{ is a } K \text{-minimal frame} \iff N(T) = \{0\} \iff N(T_1) \cap N(T_2) = \{0\}.
\]
\end{proof}
The following result provides a sufficient condition for \( \{u_i \oplus v_i\}_{i \in I} \) to be a \( K_1 \oplus K_2 \)-minimal frame for the super Hilbert space \( \mathcal{H}_1 \oplus \mathcal{H}_2 \).
\begin{theorem}
Let \( K_1 \in B(\mathcal{H}_1) \) and \( K_2 \in B(\mathcal{H}_2) \). Let \( \{u_i\}_{i \in I} \) be a \( K_1 \)-frame for \( \mathcal{H}_1 \) and \( \{v_i\}_{i \in I} \) be a \( K_2 \)-frame for \( \mathcal{H}_2 \). Let \( \theta_1 \) and \( \theta_2 \) be the frame transforms of \( \{u_i\}_{i \in I} \) and \( \{v_i\}_{i \in I} \), respectively.

If \(\overline{ R(\theta_1)} = R(\theta_2)^\perp \), then \( \{u_i \oplus v_i\}_{i \in I} \) is a \( K_1 \oplus K_2 \)-minimal frame for \( \mathcal{H}_1 \oplus \mathcal{H}_2 \).
\end{theorem}
\begin{proof}
Assume that \(\overline{ R(\theta_1)} = R(\theta_2)^\perp \), then  \( R(\theta_1) \perp R(\theta_2) \) and \( R(\theta_1)^\perp \cap R(\theta_2)^\perp = \{0\} \). Then, Theorem \ref{thm8} implies that \( \{u_i \oplus v_i\}_{i \in I} \) is a \( K_1 \oplus K_2 \)-frame for \( \mathcal{H}_1 \oplus \mathcal{H}_2 \). Since \( N(\theta_1^*) = R(\theta_1)^\perp \) and \( N(\theta_2^*) = R(\theta_2)^\perp \), we have \( N(\theta_1^*) \cap N(\theta_2^*) = \{0\} \). This means that \( N(T_1) \cap N(T_2) = \{0\} \), where \( T_1 \) and \( T_2 \) are the synthesis operators of \( \{u_i\}_{i \in I} \) and \( \{v_i\}_{i \in I} \), respectively. Hence, by Proposition \ref{prop9}, \( \{u_i \oplus v_i\}_{i \in I} \) is a \( K_1 \oplus K_2 \)-minimal frame for \( \mathcal{H}_1 \oplus \mathcal{H}_2 \).
\end{proof}

\section{$K_1$-duality, $K_2$-duality and $K_1\oplus K_2$-duality}
In this section, $\mathcal{H}_1$ and $\mathcal{H}_2$ are two right quaternionic Hilbert spaces, $K_1\in \mathbb{B}(\mathcal{H}_1)$ and $K_2\in \mathbb{B}(\mathcal{H}_2)$. Given $\{u_i\oplus v_i\}_{i\in I}$, a $K_1\oplus K_2$-frame for $\mathcal{H}_1\oplus \mathcal{H}_2$. We will explore the relationship between the $K_1$-dual frames to $\{u_i\}_{i\in I}$, the $K_2$-dual frames to $\{v_i\}_{i\in I}$ and the $K_1\oplus K_2$-dual frames to $\{u_i\oplus v_i\}_{i\in I}$.
\begin{theorem}
Let \( K_1 \in \mathbb{B}(\mathcal{H}_1) \) and \( K_2 \in \mathbb{B}(\mathcal{H}_2) \) and \( \{u_i \oplus v_i\}_{i \in I} \) be a \( K_1 \oplus K_2 \)-frame for \( \mathcal{H}_1 \oplus \mathcal{H}_2 \). 
If \( \{a_i \oplus b_i\}_{i \in I} \) is a \( K_1 \oplus K_2 \)-dual frame to \( \{u_i \oplus v_i\}_{i \in I} \), then \( \{a_i\}_{i \in I} \) is a \( K_1 \)-dual frame to \( \{u_i\}_{i \in I} \) and \( \{b_i\}_{i \in I} \) is a \( K_2 \)-dual frame to \( \{v_i\}_{i \in I} \).
\end{theorem}
\begin{proof}
Let \( \{a_i \oplus b_i\}_{i\in I} \in \mathcal{H}_1 \oplus \mathcal{H}_2 \) be a \( K_1 \oplus K_2 \)-dual frame to \( \{u_i \oplus v_i\}_{i \in I} \). Then for all \( u \oplus v \in \mathcal{H}_1 \oplus \mathcal{H}_2 \),

$$\begin{array}{rcl}
K_1 \oplus K_2 (u \oplus v) &=& \displaystyle{\sum_{i \in I}  (u_i \oplus v_i)\langle a_i \oplus b_i, u \oplus v \rangle}\\
&=&\displaystyle{ \sum_{i \in I}  u_i \langle a_i \oplus b_i, u \oplus v \rangle \oplus \sum_{i \in I}  v_i\langle a_i \oplus b_i, u \oplus v \rangle}.
\end{array}$$

Thus, for all \( u \in \mathcal{H}_1 \) and \( v \in \mathcal{H}_2 \),

\[
\begin{cases}
K_1(u) = \displaystyle{\sum_{i \in I} u_i\langle a_i \oplus b_i, u \oplus v \rangle}, \\
K_2(v) = \sum_{i \in I}  v_i\langle a_i \oplus b_i, u \oplus v \rangle.
\end{cases}
\]

In particular, by taking \( v = 0 \) in the first equality and \( u = 0 \) in the second one, we obtain:

\[
\begin{cases}
K_1(u) =\displaystyle{ \sum_{i \in I}  u_i\langle a_i, u \rangle}, \\
K_2(v) = \displaystyle{\sum_{i \in I}  v_i \langle b_i, v \rangle}.
\end{cases}
\]

Hence, \( \{a_i\}_{i \in I} \) is a \( K_1 \)-dual frame to \( \{u_i\}_{i \in I} \) and \( \{b_i\}_{i \in I} \) is a \( K_2 \)-dual frame to \( \{v_i\}_{i \in I} \).
\end{proof}
One might wonder whether the converse of the above proposition is true. The following result demonstrates that this is not always the case and provides a necessary and sufficient condition for when it does hold.
\begin{theorem}
Let \( K_1 \in \mathbb{B}(\mathcal{H}_1) \) and \( K_2 \in \mathbb{B}(\mathcal{H}_2) \). Let \( \{u_i\}_{i \in I} \) be a \( K_1 \)-frame for \( \mathcal{H}_1 \) whose \( \{a_i\}_{i \in I} \) is a \( K_1 \)-dual frame, and let \( \{v_i\}_{i \in I} \) be a \( K_2 \)-frame for \( \mathcal{H}_2 \) whose \( \{b_i\}_{i \in I} \) is a \( K_2 \)-dual frame. Then the following statements are equivalent:
\begin{enumerate}
    \item \( \{u_i \oplus v_i\}_{i \in I} \) is a \( K_1 \oplus K_2 \)-frame for \( \mathcal{H}_1 \oplus \mathcal{H}_2 \) and \( \{a_i \oplus b_i\}_{i \in I} \) is a \( K_1 \oplus K_2 \)-dual frame.
    \item \( T_2 \theta_1 = 0_{\mathcal{H}_1} \) and \( T_1 \theta_2 = 0_{\mathcal{H}_2} \).
\end{enumerate}
Where \( T_1 \) and \( T_2 \) are the synthesis operators of \( \{u_i\}_{i \in I} \) and \( \{v_i\}_{i \in I} \) respectively, and \( \theta_1 \) and \( \theta_2 \) are the frame transforms of \( \{a_i\}_{i \in I} \) and \( \{b_i\}_{i \in I} \) respectively.
\end{theorem}
\begin{proof}
Assume that \( \{u_i \oplus v_i\}_{i \in I} \) is a \( K_1 \oplus K_2 \)-frame for \( \mathcal{H}_1 \oplus \mathcal{H}_2 \) and \( \{a_i \oplus b_i\}_{i \in I} \) is a \( K_1 \oplus K_2 \)-dual frame. Then for all \( u \oplus v \in \mathcal{H}_1 \oplus \mathcal{H}_2 \), we have:

\[
K_1 \oplus K_2 (u \oplus v) = \sum_{i \in I}  (u_i \oplus v_i)\langle a_i \oplus b_i,u \oplus v \rangle.
\]

Hence, for all \( u \oplus v \in \mathcal{H}_1 \oplus \mathcal{H}_2 \),

\[
K_1(u) \oplus K_2(v) = \sum_{i \in I}  u_i\langle a_i, u \rangle + \sum_{i \in I}  u_i\langle b_i,b\rangle \oplus \sum_{i \in I}  v_i \langle a_i,u \rangle + \sum_{i \in I}  v_i \langle b_i \rangle.
\]

That means that for all \( u \oplus v \in \mathcal{H}_1 \oplus \mathcal{H}_2 \),

\[
K_1(u) \oplus K_2(v) = K_1(u) + \sum_{i \in I}  u_i \langle b_i,v \rangle \oplus \sum_{i \in I}  v_i\langle a_i,u \rangle + K_2(v).
\]

Then for all \( u \oplus v \in \mathcal{H}_1 \oplus \mathcal{H}_2 \),

\[
\begin{cases}
\displaystyle{\sum_{i \in I}  u_i\langle b_i,v \rangle = 0}, \\
\displaystyle{\sum_{i \in I}  v_i\langle a_i,u \rangle = 0}.
\end{cases}
\]

Thus, for all \( u \oplus v \in \mathcal{H}_1 \oplus \mathcal{H}_2 \), \( T_1 \theta_2(v) = 0 \) and \( T_2 \theta_1(u) = 0 \).

Conversely, assume that \( T_2 \theta_1 = 0_{\mathcal{H}_1} \) and \( T_1 \theta_2 = 0_{\mathcal{H}_2} \). Then for all \( u \oplus v \in \mathcal{H}_1 \oplus \mathcal{H}_2 \),

\[
\begin{cases}
\displaystyle{\sum_{i \in I} \langle a_i,u \rangle v_i = 0}, \\
\displaystyle{\sum_{i \in I} \langle b_i,v \rangle u_i = 0}.
\end{cases}
\]

On the other hand, we have for all \( u \oplus v \in \mathcal{H}_1 \oplus \mathcal{H}_2 \),

\[
K_1 \oplus K_2(u \oplus v) = K_1(u) \oplus K_2(v) = \sum_{i \in I}  u_i\langle  a_i, u \rangle \oplus \sum_{i \in I}  v_i\langle b_i,v \rangle
\]

\[
= \sum_{i \in I}  u_i\langle a_i,u \rangle + \sum_{i \in I}  u_i\langle b_i,v \rangle \oplus \sum_{i \in I}  v_i\langle a_i,u \rangle + \sum_{i \in I}  v_i\langle b_i,v \rangle
\]

\[
= \sum_{i \in I}  (u_i \oplus v_i)\langle a_i \oplus b_i,u \oplus v \rangle = \sum_{i \in I}  (u_i \oplus v_i)\langle  a_i \oplus b_i, u\oplus v \rangle.
\]

Hence, \( \{u_i \oplus v_i\}_{i \in I} \) is a \( K_1 \oplus K_2 \)-frame for \( \mathcal{H}_1 \oplus \mathcal{H}_2 \) and \( \{a_i \oplus b_i\}_{i \in I} \) is a \( K_1 \oplus K_2 \)-dual frame.
\end{proof}

\medskip

	\section*{Acknowledgments}
	It is mu great pleasure to thank the referee for his careful reading of the paper and for several helpful suggestions.
	
	\section*{Ethics declarations}
	
	\subsection*{Availablity of data and materials}
	Not applicable.
	\subsection*{Conflict of interest}
	The author declares that he has no competing interests.
	\subsection*{Fundings}
	Not applicable.


\begin{thebibliography}{99}

\bibitem{1}
Adler, S.L.: Quaternionic Quantum Mechanics and Quantum Fields. Oxford University Press, New York (1995).

\bibitem{2}
Cahill, J., Casazza, P.G., Li, S.: Non-orthogonal fusion frames and the sparsity of fusion frame operators. \emph{J. Fourier Anal. Appl.} 18, 287--308 (2012).

\bibitem{3}
Charfi, S., Ellouz, H.: Frame of exponentials related to analytic families operators and application to a non-self adjoint problem of radiation of a vibrating structure in a light fluid. \emph{Complex Anal. Oper. Theory} 13, 839--858 (2019).
\bibitem{4}
Christensen, O.: An Introduction to Frames and Riesz Bases. Applied and Numerical Harmonic Analysis, 2nd edn. Birkhäuser, Basel (2016).

\bibitem{5}
Colombo, F., Gantner, J., Kimsey, David P.: Spectral Theory on the S-Spectrum for Quaternionic Operators. \emph{Operator Theory: Advances and Applications}, 270, p. ix+356. Birkhäuser, Cham (2018).

\bibitem{6}
Daubechies, I., Grossmann, A., Meyer, Y.: Painless nonorthogonal expansions. \emph{J. Math. Phys.} 24, 1271--1283 (1986).

\bibitem{7}
Duffin, R.J., Schaeffer, A.C.: A class of nonharmonic Fourier series. \emph{Trans. Am. Math. Soc.} 72, 341--366 (1952).

\bibitem{8}
Ellouz, H., Feki, I., Jeribi, A.: On a Riesz basis of exponentials related to the eigenvalues of an analytic operator and application to a non-selfadjoint problem deduced from a perturbation method for sound radiation. \emph{J. Math. Phys.} 54, 112101 (2013).

\bibitem{9} 
Ellouz, H. Some Properties of K-Frames in Quaternionic Hilbert Spaces. Complex Anal. Oper. Theory 14, 8 (2020).

\bibitem{10}
Gavruţa, L.: Frames for operators. \emph{Appl. Comput. Harmon. Anal.} 32, 139--144 (2012).

\bibitem{11}
Ghiloni, R., Moretti, V., Perotti, A.: Continuous slice functional calculus in quaternionic Hilbert spaces. \emph{Rev. Math. Phys.} 25, 1350006 (2013).

\bibitem{12}
Guo, X.: Canonical dual K-Bessel sequences and dual K-Bessel generators for unitary systems of Hilbert spaces. \emph{J. Math. Anal. Appl.} 444, 598--609 (2016).

\bibitem{13}
Jeribi, A.: Denseness, Bases and Frames in Banach Spaces and Applications. De Gruyter, Berlin (2018).

\bibitem{14}
Jia, M., Zhu, Y.-C.: Some results about the operator perturbation of a K-frame. \emph{Results Math.} 73(4), 138 (2018).

\bibitem{15}
Sharma, S.K., Goel, S.: Frames in quaternionic Hilbert spaces. \emph{J. Math. Phys. Anal. Geom.} 15(3), 395--411 (2019).

\bibitem{16}
Xiao, X., Zhu, Y., Gavruţa, L.: Some properties of K-frames in Hilbert spaces. \emph{Results Math.} 63, 1243--1255 (2013).

\bibitem{17}
Young, R.M.: An Introduction to Nonharmonic Fourier Series. Academic Press, London (1980).

\end{thebibliography}
\end{document}